\documentclass[a4paper]{article}

\usepackage{amsmath}
\usepackage{amssymb}
\usepackage{dsfont}
\usepackage{ifthen}
\usepackage{xparse}

\newcommand{\compl}[1]{#1^c} 
\renewcommand{\div}{\mathrm{div}\ } 

\renewcommand{\epsilon}{\varepsilon}

\let\foralltemp\forall
\renewcommand{\forall}{\foralltemp\, }

\DeclareDocumentCommand{\integral}{ O{} O{} m O{x}}{\int_{#1}^{#2} #3\, \mathrm{d}#4} 

\newcommand{\Natural}{\mathbb{N}}
\newcommand{\norm}[2][]{\| #2 \|_{#1}} 

\newcommand{\Real}{\mathbb{R}} 

\newcommand{\set}[2]{\left\{ #1 : #2 \right\}} 
\newcommand{\setdiff}{\backslash} 
\newcommand{\simpleset}[1]{\left\{ #1 \right\}} 
\newcommand{\unitsphere}[1]{\mathbb{S}^{#1-1}}

\newcounter{object}[section]
\renewcommand{\theobject}{\arabic{section}.\arabic{object}}

\newenvironment{definition}[1][]{\rmfamily \medskip%
			  \refstepcounter{object}%
			  \noindent \textbf{Definition \theobject%
			    \ifthenelse{\equal{#1}{}}{.}{\ (#1).}\ }}{\medskip}

\newenvironment{corollary}[1][]{\medskip%
			 \refstepcounter{object}%
			 \noindent \textbf{Corollary \theobject%
			 \ifthenelse{\equal{#1}{}}{.}{\ (#1).}\ }\itshape }{\medskip}
\newenvironment{example}{\rmfamily \medskip%
			  \refstepcounter{object}%
			  \noindent \textbf{Example \theobject.}\ }{\hfill $\blacktriangle$\medskip}
\newenvironment{lemma}{\rmfamily \medskip%
			  \refstepcounter{object}%
			  \noindent \textbf{Lemma \theobject .}\ \itshape}{\medskip}
\newenvironment{proposition}[1][]{\rmfamily \medskip%
			  \refstepcounter{object}%
			  \noindent \textbf{Proposition \theobject%
			 \ifthenelse{\equal{#1}{}}{.}{\ (#1).}\ }\itshape }{\medskip}
\newenvironment{proof}[1][]{
  \noindent \textit{Proof\ifthenelse{\equal{#1}{}}{}{ of #1}.}\ }{\hfill $\blacksquare$ \medskip}
\newenvironment{remark}{\rmfamily \medskip%
			  \refstepcounter{object}%
			  \noindent \textbf{Remark \theobject.}\ }{\hfill $\blacktriangle$\medskip}
\newenvironment{theorem}[1][]{\medskip%
			 \refstepcounter{object}%
			 \noindent \textbf{Theorem \theobject%
			 \ifthenelse{\equal{#1}{}}{.}{\ (#1).}\ }\itshape }{\medskip}

\usepackage{fullpage}
\usepackage{enumitem}
\usepackage{xcolor}

\newcommand{\curlyE}{\mathcal E}
\newcommand{\Per}{\mathrm{Per}}

\title{The anisotropic fractional isoperimetric problem with respect to unconditional unit balls}
\author{Andreas Kreuml}
\date{}

\begin{document}

\maketitle

\begin{abstract}
  The minimizers of the anisotropic fractional isoperimetric inequality with respect to the convex body $K$ in $\Real^n$ are shown to be equivalent to star bodies whenever $K$ is strictly convex and unconditional.
  From this a P\'olya-Szeg{\H o} principle for anisotropic fractional seminorms is derived by using symmetrization with respect to star bodies.

\end{abstract}

\section{Introduction}

Let $K \subset \Real^n$ be an origin-symmetric convex body and $s \in (0,1)$.
The anisotropic fractional $s$-perimeter was defined by Ludwig \cite{ludwig-perimeter} for Borel sets $E \subseteq \Real^n$ by
\begin{equation*}
	P_s(E,K) := \integral[E]{
		\integral[\compl E]{
			\frac 1 {\norm[K]{x-y}^{n+s}}
		}[y]
	},
\end{equation*}
where $\compl E$ is the complement of $E$ in $\Real^n$ and $\norm[K]{\cdot}$ is the norm on $\Real^n$ with closed unit ball $K$, i.e. $K = \set{x \in \Real^n}{\norm[K]{x} \le 1}$.
Here, we say that a set $K \subset \Real^n$ is a convex body if it is compact, convex, and it has non-empty interior.
The isotropic case, that is, $K = B$ is the Euclidean unit ball, leads to the (Euclidean) fractional perimeter (denoted by $P_s(E)$) which is closely connected to the theory of fractional Sobolev spaces and has been extensively studied over the last two decades (see \cite{ambrosio-gamma, bbmlimiting, bucurvaldinoci, crs, dicastro, frankseiringer, hurri, kreumlmord, xiao} and the references therein).
In particular, for bounded Borel sets $E \subset \Real^n$ the fractional isoperimetric inequality
\begin{equation*}
  P_s(E) \ge \gamma_{n,s}|E|^{\frac{n-s}{n}}
\end{equation*}
holds with sharp constant $\gamma_{n,s} > 0$ and there is equality precisely for sets equivalent to Euclidean balls (see \cite{frankseiringer}).

Anisotropic fractional perimeters share many properties with Euclidean fractional perimeters (see Section \ref{sec:perimeters}) and also fulfil an isoperimetric inequality,
\begin{equation}
  \label{eq:anfracisoineq}
  P_s(E,K) \ge \gamma_{n,s}(K)|E|^{\frac{n-s}{n}},
\end{equation}
where $E \subset \Real^n$ is a bounded Borel set and $\gamma_{n,s}(K) > 0$ is the optimal constant.
From the Fr\'echet-Kolmogorov compactness criterion it can be deduced that there exists a minimizer of $\eqref{eq:anfracisoineq}$.
However, if $K \neq B$, little is known about the value of $\gamma_{n,s}(K)$ and the equality cases.
By a result of Ludwig \cite{ludwig-perimeter}, the minimizers of \eqref{eq:anfracisoineq} are in general not homothetic to the unit ball $K$ which is a striking difference to the anisotropic isoperimetric inequality (cf. \cite{taylor})
\begin{equation*}
  P(E,K) \ge n|K|^{\frac 1 n}|E|^{\frac{n-1}n}
\end{equation*} 
for Borel sets $E \subset \Real^n$ with $|E| < \infty$,
where all minimizers are homothetic to $K$ up to sets of measure 0.
Here, $P(E,K)$ is the anisotropic perimeter with respect to $K$ as defined in Section \ref{sec:perimeters}.

In our first main theorem we show that under certain symmetry assumptions on $K$ all minimizers must be star-shaped.

\begin{theorem}
  \label{th:introshape}
	  Let $K \subset \Real^n$ be an unconditional strictly convex body.
	  Then every minimizer $M \subset \Real^n$ of the anisotropic fractional isoperimetric inequality \eqref{eq:anfracisoineq} is up to translation equivalent to an unconditional star body.
\end{theorem}

Here, we say that two sets are equivalent if they differ on a set of measure 0.
A subset of $\Real^n$ is called unconditional if it symmetric with respect to every coordinate hyperplane, and called a star body if it is star-shaped with respect to the origin $o$ and its radial function is strictly positive and continuous on the Euclidean unit sphere $\unitsphere n$ (see also Section \ref{sec:symm})
Furthermore, a convex body $K \subset \Real^n$ is called strictly convex if for all $x,y \in K$ with $x \neq y$ and $0 < \lambda < 1$ it holds that $(1-\lambda)x+\lambda y \in \text{int } K$, where $\text{int } K$ denotes the interior of $K$.

A natural idea for deducing properties of minimizers is to show that anisotropic fractional perimeters do not increase under symmetrization where the rearranged sets have the same symmetries as the unit ball $K$.
Indeed, for Euclidean fractional perimeters, Riesz's rearrangement inequality for Schwarz symmetrization, where all symmetrized sets are Euclidean balls, yields that all minimizers are equivalent to Euclidean balls.
However, Van Schaftingen \cite{vanschaft} showed that a corresponding rearrangement inequality does not hold true when the symmetrization is taken with respect to a unit ball $K$ different from $B$.
So instead of symmetrizing with respect to $K$, we use Steiner symmetrization with respect to a hyperplane of symmetry of $K$ and are still able to apply Riesz's rearrangement inequality in this situation. 
Although we focus on anisotropic fractional perimeters in this article, we remark that this technique can also be applied to general nonlocal perimeter functionals whose kernel functions satisfy suitable symmetry assumptions (see Section \ref{sec:perimeters}).

The second main theorem deals with an anisotropic fractional P\'olya-Szeg{\H o} principle in which an\-iso\-tro\-pic fractional seminorms before and after symmetrization are compared.
Following an anisotropic version obtained by Alvino et al. \cite{alvino} which is described in Section \ref{sec:polya} we take the symmetrization with respect to the minimizers of the anisotropic isoperimetric inequality \eqref{eq:anfracisoineq}.
For this we need to extend the notion of anisotropic symmetrization to star bodies (see Section \ref{sec:symm}).
The anisotropic fractional P\'olya-Szeg{\H o} principle then reads as follows.

\begin{theorem}
  \label{th:introps}
	Let $K \subset \Real^n$ be an unconditional strictly convex body and $M$ a minimizer of \eqref{eq:anfracisoineq}.
	Then the anisotropic rearrangement $f^M$ with respect to $M$ is well-defined, and
	\begin{equation}
		\integral[\Real^n]{
			\integral[\Real^n]{
				\frac{|f(x)-f(y)|}{\norm[K]{x-y}^{n+s}}
			}[y]
		}[x] \ge \integral[\Real^n]{
			\integral[\Real^n]{
				\frac{|f^M(x)-f^M(y)|}{\norm[K]{x-y}^{n+s}}
			}[y]
		}[x].
	\end{equation}
	for all $f \in L^1(\Real^n)$.
\end{theorem}

The paper is structured as follows:
\\

In the beginning of Section 2 we fix some basic notation.
Then we recall results on symmetrization, anisotropic fractional perimeters and P\'olya-Szeg{\H o} principles to put the main results into context and state results needed in the following sections.

In Section 3 we show general rearrangement inequalities for symmetric decreasing rearrangement and Steiner symmetrization.
We use them to derive a P\'olya-Szeg{\H o} principle for Steiner symmetrization of anisotropic fractional seminorms, as well as a Steiner inequality for anisotropic fractional perimeters, whenever the unit ball is symmetric with respect to some coordinate hyperplane.

In Section 4 we give the proofs of Theorem \ref{th:introshape} and \ref{th:introps} with the help of the Steiner inequality obtained in Section 3.

\section{Background material}

We always assume $n \in \Natural$ and $n \ge 1$.
For $x = (x_1,\dots,x_n) \in \Real^n$ and $y = (y_1,\dots,y_n) \in \Real^n$ we denote by $|x| := \left( \sum_{i=1}^n x_i^2 \right)^{\frac 1 2}$ the Euclidean norm of $x$ and by $x \cdot y := \sum_{i=1}^n x_i y_i$ their inner product.
The closed unit ball of the Euclidean norm is the Euclidean unit ball $B := \set{x \in \Real^n}{|x| \le 1}$ and its boundary the Euclidean unit sphere $\unitsphere n := \set{x \in \Real^n}{|x| = 1}$.

The characteristic function of a set $E \subseteq \Real^n$ is the function $\chi_E: \Real^n \to \simpleset{0,1}$ with $\chi_E(x) = 1$ if $x \in E$ and $\chi_E(x) = 0$ otherwise.
If $E \subseteq \Real^n$ is a Borel set, its ($n$-dimensional Lebesgue) measure is denoted either by $|E|$ or $\mathcal L^n (E)$ to emphazise the dimension.

If $K \subset \Real^n$ is an origin-symmetric convex body, then its polar body
\begin{equation*}
  K^\circ := \set{y \in \Real^n}{x \cdot y \le 1 \text{ for all } x \in K}
\end{equation*}
is again an origin-symmetric convex body.

\subsection{Symmetrization}
\label{sec:symm}

In this section we extend the notion of anisotropic symmetrization (cf. \cite{alvino}, \cite{vanschaft}) to star-shaped sets.
For a general reference to star-shaped sets and bodies we refer to the books of Gardner \cite{gardner} and Schneider \cite{schneider}.

A set $L \subseteq \Real^n$ is called star-shaped (with respect to the origin $o$) if for every $x \in L$ the line segment $[o,x] := \set{\lambda x}{0 \le \lambda \le 1}$ connecting the origin $o$ with $x$ lies entirely in $L$.
If $L$ is bounded and star-shaped then its radial function $\rho_L: \Real^n \setdiff \{o\} \to [0,\infty)$ is defined by
\begin{equation*}
  \rho_L(x) := \sup \set{\lambda \ge 0}{\lambda x \in L}.
\end{equation*}
Since radial functions are positively homogeneous of degree $-1$, i.e. for every $x \in \Real^n \setdiff \{o\}$ and $\lambda > 0$
\begin{equation*}
  \rho_L(\lambda x) = \lambda^{-1} \rho_L(x),
\end{equation*}
they are completely determined by their values on the Euclidean unit sphere $\unitsphere n$.
We call a bounded star-shaped set $L \subset \Real^n$ a \emph{star body} if it contains the origin in its interior and its radial function is continuous.

\begin{definition}
  \label{defanisosym}
  Let $L \subset \Real^n$ be a star body.
    Then the (anisotropic) symmetrization $E^L$ of the set $E \subseteq \Real^n$ with respect to $L$ is defined as follows:
    If $|E| = \infty$, then $E^L := \Real^n$.
    If $|E| < \infty$, then
    \begin{equation*}
      E^L :=  r L
    \end{equation*}
    where $rL = \set{r \ell}{\ell \in L}$ and $r \ge 0$ is chosen such that $|E^L| = |E|.$
\end{definition}

Note that in case $|E| < \infty$ the factor $r \ge 0$ is uniquely determined by the relation $|E^L| = r^n |L| = |E|$.
Since $L$ has a continuous radial function bounded away from $0$ on $\unitsphere n$, every point $x \in \Real^n$ lies on the boundary of precisely one of the dilates $rL$ with $r \ge 0$.
Furthermore, this notion of symmetrization does not depend on the scaling of $L$, i.e. if $\tilde L = \lambda L$ for $\lambda > 0$, then $E^{\tilde L} = E^L$.

\begin{example}
  \begin{enumerate}
    \item
      If $L$ is an origin-symmetric convex body, then the symmetrization with respect to $L$ was introduced by Alvino et al. \!\!\!\! \cite{alvino} under the name of convex symmetrization and extended to non-symmetric convex bodies by Van Schaftingen \cite{vanschaft}.
    \item
      Symmetrization with respect to $L = B$, the Euclidean unit ball, is called \emph{Schwarz symmetrization} and denoted by $\cdot^*$, i.e. $E^*=E^B$.
       For the decomposition $\Real^n = \Real^{n-1} \times \Real$ we write $x \in \Real^n$ as $x = (x',x_n)$ with $x' \in \Real^{n-1}$ and $x_n \in \Real$.
       If $A \subseteq \Real^n$ and $x' \in \Real^{n-1}$, the section $A_{x'}$ is defined as
       \begin{equation*}
	 A_{x'} := \set{y \in \Real}{(x',y) \in A}.
       \end{equation*}
       The \emph{Steiner symmetrization} $A^{\#}$ of $A$ with respect to the hyperplane $\simpleset{x_n=0}$ (or simply with respect to $x_n$) is then defined by
       \begin{equation*}
	 [A^{\#}]_{x'} = [A_{x'}]^*,
       \end{equation*}
       for every $x' \in \Real^{n-1}$, where $[A_{x'}]^*$ is the Schwarz symmetrization of the set $A_{x'}$ in $\Real$.
  \end{enumerate}
  \vspace{-7mm}
\end{example}

In the following, if $f: A \to \Real$ is a function on $A \subseteq \Real^n$ and $\tau \in \Real$, we write
\begin{equation*}
  \simpleset{f > \tau} = \set{x \in A}{f(x)>\tau}
\end{equation*}
for the level sets of $f$.

\begin{definition}
  Let $L \subset \Real^n$ be a star body and $f: \Real^n \to \Real$ a measurable function such that all level sets $\{|f| > \tau\}$ for $\tau > 0$ have finite measure.
  Then the (anisotropic) symmetrization $f^L : \Real^n \to [0,\infty)$ of $f$ with respect to $L$ is defined as
    \begin{equation*}
      f^L(x) := \sup \set{\tau > 0}{x \in \simpleset{|f|>\tau}^L},
    \end{equation*}
    where $\simpleset{|f| > \tau}^L$ is the symmetrization of the set $\simpleset{|f|>\tau}$ with respect to $L$.
\end{definition}

Again, symmetrization of functions with respect to $L$ does not depend on the scaling on $L$.

\begin{example}
  \label{ex:fctrearr}
  In the case of Schwarz symmetrization, $f^*$ is also commonly known as the \emph{symmetric decreasing rearrangement} of $f$ (cf. \cite{liebloss}).
  For $x' \in \Real^{n-1}$ we define the section $f_{x'}: \Real \to \Real$ of $f$ as
  \begin{equation*}
    f_{x'}(y) := f(x',y).
  \end{equation*}
  Then the Steiner symmetrization $f^{\#}$ of a function $f$ with respect to the hyperplane $\simpleset{x_n=0}$ is defined by
\begin{equation*}
  f^{\#} (x',x_n) := \sup \set{\tau > 0}{ x_n \in \set{y \in \Real}{f(x',y) > \tau}^*},
\end{equation*}
for $x = (x',x_n) \in \Real^{n-1} \times \Real$, i.e. $[f^{\#}]_{x'} = [f_{x'}]^*$.
\end{example}

The next result shows that the level sets of a symmetrized function $f^L$ are obtained by symmetrizing the corresponding level sets of $f$.
It is well-known for symmetric decreasing rearrangement (see e.g. \cite[Chapter 3.3]{liebloss}) and the proof for symmetrization with respect to star-shaped bodies follows along the same lines.
Since it is short we include it.

\begin{proposition}
  Let $f:\Real^n \to \Real$ be measurable with $|\simpleset{|f|>\tau}|$ finite for all $\tau > 0$. 
  Then 
  \begin{equation*}
    \simpleset{f^L > \tau} = \simpleset{|f| > \tau}^L
  \end{equation*}
for all $\tau > 0$.
\end{proposition}

\begin{proof}
  First note that from $f^L(x) = \integral[0][\infty]{\chi_{\simpleset{|f|>s}^L}(x)}[s] > \tau$ and $\simpleset{|f|> s_1}^L \supseteq \simpleset{|f| > s_2}^L$ for $s_1 \le s_2$ it follows that $x \in \simpleset{|f| > \tau}^L$.

  For the other direction we note that the distribution function $s \mapsto |\simpleset{|f|>s}|$ is continuous from the right, so $x \in \simpleset{|f|>\tau}^L$ implies that $x \in \simpleset{|f|>\tau+\delta}^L$ for some $\delta > 0$ and eventually
  \begin{equation*}
    f^L(x) = \integral[0][\infty]{\chi_{\simpleset{|f|>s}^L}(x)}[s] \ge \tau+\delta,
  \end{equation*}
  so $x \in \simpleset{f^L > \tau}$.
\end{proof}

We will use the following strict version of Riesz's rearrangement inequality (cf. \cite{lieb}):

\begin{theorem}[Riesz's rearrangement inequality]
  \label{th:riesz}
  Let $f,g$ and $k$ be non-negative measurable functions on $\Real^n$ such that all their level sets have finite measure.
    Then,
    \begin{equation}
      \label{eq:rieszrearr}
      \integral[\Real^n]{
	\integral[\Real^n]{
	  f(x) k(x-y) g(y)
	}[y]
      } \le \integral[\Real^n]{
	\integral[\Real^n]{
	f^*(x) k^*(x-y) g^*(y)
	}[y]
      },
    \end{equation}
    where $\cdot^*$ denotes symmetric decreasing rearrangement (as introduced in Example \ref{ex:fctrearr}).

    Furthermore, if $k$ is strictly symmetric decreasing, i.e. $k(x) > k(y)$ whenever $|x| < |y|$, then equality holds in \eqref{eq:rieszrearr} if and only if there exists $c \in \Real^n$ such that $f(x) = f^*(x-c)$ and $g(x) = g^*(x-c)$ almost everywhere.
\end{theorem}

We conclude this section with a result by Van Schaftingen \cite{vanschaft} that Riesz's rearrangement inequality is in general not true, if Schwarz symmetrization is replaced by symmetrization with respect to a unit ball different from $B$.

\begin{theorem}[\cite{vanschaft}]
  Let $K$ be a convex body with $o \in \mathrm{int}\, K$.
  If
  \begin{equation*}
    \integral[\Real^n]{
      \integral[\Real^n]{
	f(x) k(x-y) g(y)
      }[y]
    } \le \integral[\Real^n]{
      \integral[\Real^n]{
	f^K(x) k^K(x-y) g^K(y)
      }[y]
    }
  \end{equation*}
  for all non-negative continuous functions $f,g$ and $k$ with compact support, then $K = B$.
\end{theorem}



\subsection{Anisotropic fractional perimeters}
\label{sec:perimeters}

We denote by $C_c^1(\Real^n;\Real^n)$ the space of all compactly supported and continuously differentiable vector fields $T: \Real^n \to \Real^n$.
We first recall the definition of perimeter and its anisotropic version.

\begin{definition}
  Let $E \subseteq \Real$ be a Borel set.
  \begin{enumerate}
    \item
      The \emph{(Euclidean) perimeter} of $E$ is defined by
      \begin{equation*}
	P(E):= \sup \set{\integral[E]{\div T}}{T \in C_c^1(\Real^n;\Real^n), |T| \le 1}.
      \end{equation*}
    \item
      Let $K \subset \Real^n$ be the closed unit ball of the norm $\norm[K]{\cdot}$.
      The \emph{anisotropic perimeter} of $E$ with respect to $K$ (cf. \cite{amar}) is defined by
      \begin{equation*}
	P(E,K):= \sup \set{\integral[E]{\div T}}{T \in C_c^1(\Real^n;\Real^n), \norm[K]{T} \le 1}.
      \end{equation*}
  \end{enumerate}
\end{definition}

If $K=B$ is chosen in the definition of the anisotropic perimeter, then we recover the standard perimeter.
In this sense, the anisotropic perimeter can be understood as a generalization of its Euclidean counterpart.
Anisotropic fractional and Euclidean fractional perimeters are related the same way.

In the following, we list some properties of geometric interest which all perimeter functionals we have presented so far have in common.
For their proofs we refer to \cite{maggi}, and \cite{cesaroni} for the fractional versions.
To provide a simple unified notation for these functionals, let $\mathcal P_s$ denote the anisotropic fractional perimeter with respect to $K$ if $s \in (0,1)$ and the anisotropic perimeter with respect to $K$ if $s=1$\footnote{The integrals in the definition of the anisotropic fractional perimeter $P_s(E,K)$ do not converge for $s = 1$, unless $E$ or $\compl E$ is a set of measure 0 (cf. \cite{brezis}), which justifies the need for the new notation $\mathcal P_s$. However, Theorem \ref{th:anisobbm} shows that the anisotropic perimeter is the endpoint in the scale of anisotropic fractional perimeters in a certain sense.}.
Let $E \subseteq \Real^n$ be a Borel set.
Then,
\begin{itemize}
  \item $\mathcal P_s (E) = \mathcal P_s (\compl E)$,
  \item $\mathcal P_s$ is invariant under translations, i.e. if $y \in \Real^n$, then $\mathcal P_s(E+y) = \mathcal P_s (E)$, where $E+y := \set{x+y}{x \in E}$.

    If $K = B$, then $\mathcal P_s$ is also invariant under rotations, i.e. if $\theta \in SO(n)$ is a rotation, then $\mathcal P_s(\theta E) = \mathcal P_s(E)$, where $\theta E := \set{\theta x}{x \in E}$,
  \item
    $\mathcal P_s$ is $(n-s)$-homogeneous, i.e. if $\lambda > 0$, then $\mathcal P_s(\lambda E) = \lambda^{n-s} \mathcal P_s(E)$, where $\lambda E := \set{\lambda x}{x \in E}$,
  \item $\mathcal P_s$ is lower semicontinuous with respect to $L^1(\Real^n)$-convergence, i.e. if $\integral[\Real^n]{|\chi_{E_i}-\chi_E|} \to 0$ for Borel sets $E_i, E \subseteq \Real^n$ as $i \to \infty$, then $\displaystyle \mathcal P_s(E) \le \liminf_{i \to \infty} \mathcal P_s(E_i)$.
\end{itemize}

Anisotropic fractional perimeters are related to anisotropic perimeters by the following formula, which was first shown in \cite{bbm} and \cite{davila} in the isotropic case and extended in \cite{ludwig-perimeter} to general unit balls $K$.

\begin{theorem}[{\cite[Theorem 4]{ludwig-perimeter}}]
  \label{th:anisobbm}
  Let $E \subset \Real^n$ be a bounded Borel set of finite perimeter.
  Then
  \begin{equation*}
    \lim_{s \to 1} (1-s)P_s(E,K) = P(E,ZK),
  \end{equation*}
  where $ZK$ is the moment body of $K$ given by
  \begin{equation*}
    \norm[Z^\circ K]{v} = \frac{n+1}{2} \integral[K]{|v \cdot x|},\ v \in \Real^n.
  \end{equation*}
\end{theorem}
Here, $Z^\circ K \subset \Real^n$ is the polar body of $ZK$ defined by
\begin{equation*}
  Z^\circ K := \set{y \in \Real^n}{x \cdot y \le 1 \text{ for all } x \in ZK}.
\end{equation*}

Finding the minimizers of the anisotropic fractional isoperimetric inequality \eqref{eq:anfracisoineq} is equivalent to finding all sets for which
\begin{equation}
  \label{eq:anisofracisoprob}
  \gamma_{n,s}(K) = \inf \set{P_s(E,K)|E|^{-\frac{n-s}{n}}}{E \subset \Real^n \text{ bounded}, |E|>0}.
\end{equation}
is attained.
For the isotropic version $K = B$ the minimizers of the isoperimetric inequality \eqref{eq:anfracisoineq} are given by sets equivalent to Euclidean balls.
By equivalence of norms, there exists constants $\alpha \le \beta$ such that
\begin{equation*}
  \alpha P_s(E) \le P_s(E,K) \le \beta P_s(E)
\end{equation*}
for all Borel sets $E \subseteq \Real^n$.
Thus $0 < \gamma_{n,s}(K) < \infty$ and there exists a sequence of Borel sets $E_i$ contained in a ball $B_R \subset \Real^n, R > 0,$ with $|E_i|=m$ and $\displaystyle \gamma_{n,s}(K) = \lim_{i \to \infty} P_s(E_i,K)$.
By the Frechet-Kolmogorov compactness theorem, there exists a limit $E \subset \Real^n$ of this sequence with respect to $L_{loc}^1(B_R)$-convergence which is a minimizer of \eqref{eq:anisofracisoprob}, see \cite[(4)]{ambrosio-gamma}.

We remark that, as a consequence of Theorem \ref{th:anisobbm}, minimizers of the anisotropic fractional isoperimetric inequality \eqref{eq:anfracisoineq} are in general not homothetic to $K$ up to sets of measure 0 (cf. \cite[Theorem 7]{ludwig-perimeter}).

Closely related to fractional perimeters are fractional Sobolev spaces.
For $0<s<1$ and $1 \le p < \infty$ the fractional Sobolev seminorm of a measurable function $f: \Real^n \to \Real$ is defined as
\begin{equation}
  \label{eq:fracnormdef}
  [f]_{s,p} := \left( \integral[\Real^n]{
    \integral[\Real^n]{
      \frac{|f(x)-f(y)|^p}{|x-y|^{n+sp}}
    }[y]
  } \right)^{\frac 1 p}.
\end{equation}
The fractional Sobolev space $W^{s,p}(\Real^n)$ consists of all functions $f \in L^p(\Real^n)$ for which this seminorm is finite, i.e. $[f]_{s,p} < \infty$.
The interested reader is refered to the introductory article \cite{hitchhiker} for more information on fractional Sobolev spaces.
The anisotropic fractional seminorm, introduced by Ludwig \cite{ludwig-norm}, is obtained by replacing the Euclidean norm in \eqref{eq:fracnormdef} by an arbitrary norm $\norm[K]{\cdot}$ with unit ball $K$,
\begin{equation*}
  [f]_{s,p,K} := \left( \integral[\Real^n]{
    \integral[\Real^n]{
      \frac{|f(x)-f(y)|^p}{\norm[K]{x-y}^{n+sp}}
    }[y]
  } \right)^{\frac 1 p}.
\end{equation*}
The anisotropic fractional perimeter of a Borel set $E \subseteq \Real^n$ can be expressed in terms of seminorms by $[\chi_E]_{s,1,K} = 2 P_s(E,K)$.
On the other hand, the $W^{s,1}$-seminorm of a function can be computed by the perimeters of its level sets via the following coarea formula which was shown by Visintin \cite{visintin} (see also \cite[Lemma 10]{ambrosio-gamma} and \cite[(23)]{ludwig-perimeter}):

\begin{theorem}[generalized coarea formula]
  \label{th:coarea}
  For $f \in L^1(\Real^n)$,
  \begin{equation*}
    \integral[\Real^n]{
      \integral[\Real^n]{
	\frac{|f(x)-f(y)|}{\norm[K]{x-y}^{n+s}}
      }[y]
    } = 2 \integral[0][\infty]{P_s(\simpleset{|f|>\tau},K)}[\tau].
  \end{equation*}

\end{theorem}
Finally, let us remark that anisotropic fractional perimeters belong to the larger class of \emph{nonlocal} perimeters, as introduced in \cite{cesaroni}.
Let $k: \Real^n \to [0,\infty)$ be a measurable function such that $\min(|\cdot|,1)k \in L^1(\Real^n)$.
  Then the nonlocal perimeter $\Per_k$ is defined for Borel sets $E \subseteq \Real^n$ as
  \begin{equation*}
   \Per_k(E) := \integral[E]{
      \integral[\compl E]{
	k(x-y)
      }[y]
    }.
    \end{equation*}
    In contrast to the anisotropic fractional isoperimetric problem, it is not known in general if there exist minimizers for the nonlocal isoperimetric problem,
    \begin{equation} 
      \label{eq:nonlocalip}
      \inf \set{\Per_k(E)}{E \subset \Real^n \text{ bounded}, |E| = m},
    \end{equation}
    where $m > 0$ is fixed.
    For partial results on the existence of minimizers we refer to \cite{cesaroni}.

    We will state and prove all results in the following sections for anisotropic fractional perimeters and mention when analogous statements hold for general nonlocal perimeters.


\subsection{P\'olya-Szeg{\H o} inequalities}
\label{sec:polya}

In this section we recall some facts on P\'olya-Szeg{\H o} inequalities for different seminorms and their connection to isoperimetric inequalities.
For the classical $W^{1,p}$-seminorm, $1 \le p < \infty$, P\'olya \& Szeg{\H o} \cite{polyaszegoe} showed that if $f \in W^{1,p}(\Real^n)$, then
  \begin{equation*}
    \integral[\Real^n]{|\nabla f(x)|^p} \ge \integral[\Real^n]{|\nabla f^*(x)|^p},
  \end{equation*}
  where $f^*$ is the symmetric decreasing rearrangement of $f$.
  This inequality can be derived from the isoperimetric inequality together with the coarea formula.
  Alvino et al.\!\!\!\! \cite{alvino} introduced the notion of convex symmetrization in order to prove the anisotropic P\'olya-Szeg{\H o} inequality for $f \in W^{1,p}(\Real^n)$,
  \begin{equation*}
    \integral[\Real^n]{\norm[K^\circ]{\nabla f(x)}^p} \ge
    \integral[\Real^n]{\norm[K^\circ]{\nabla f^K(x)}^p},
  \end{equation*}
  where the symmetrization is taken with respect to $K$.
  The choice of this symmetrization stems from the equality cases of the anisotropic isoperimetric inequality $P(E,K) \ge n |K|^{\frac 1 n} |E|^{\frac{n-1} n}$ where equality holds if and only if $E$ is equivalent to a set homothetic to $K$ (see e.g. \cite{taylor}).
  
  For fractional seminorms, a P\'olya-Szeg{\H o} principle was first proved by Almgren \& Lieb \cite{almgrenlieb} and the full description of equality cases was settled by Frank \& Seiringer \cite{frankseiringer}.
  It states that if $f \in W^{s,p}(\Real^n)$, then
  \begin{equation*}
    \integral[\Real^n]{
      \integral[\Real^n]{
	\frac{|f(x)-f(y)|^p}{|x-y|^{n+sp}}
      }[y]
    } \ge 
    \integral[\Real^n]{
      \integral[\Real^n]{
	\frac{|f^*(x)-f^*(y)|^p}{|x-y|^{n+sp}}
      }[y]
    },
  \end{equation*}
  In Theorem \ref{th:seminormsteiner} of the following section we extend this inequality to Steiner symmetrization and give a full description of equality cases.

\section{Steiner symmetrization and nonlocal functionals}


The main result of this section is a P\'olya-Szeg{\H o} inequality for anisotropic fractional seminorms, where the unit ball $K$ of the norm $\norm[K]{\cdot}$ is symmetric with respect to the hyperplane $\simpleset{x_n=0}$ and the symmetrization is Steiner symmetrization with respect to the same hyperplane.

  \begin{theorem}
	  \label{th:seminormsteiner}
	  Let $s \in (0,1)$ and $1 \le p < \infty$, and let the unit ball $K$ of the norm $\norm[K]{\cdot}$ be symmetric with respect to the hyperplane $\simpleset{x_n=0}$.
	  If $f \in W^{s,p}(\Real^n)$ and $f^{\#}$ is the Steiner symmetrization of $f$ with respect to $x_n$, then $f^{\#} \in W^{s,p}(\Real^n)$ and
	  \begin{equation}
		  \label{eq:fctsteiner}
		  \integral[\Real^n]{
			  \integral[\Real^n]{
				  \frac {|f(x)-f(y)|^p} {\norm[K]{x-y}^{n+sp}}
			  }[y]
		  } \ge 
		  \integral[\Real^n]{
			  \integral[\Real^n]{
			    \frac {|f^{\#}(x)-f^{\#}(y)|^p} {\norm[K]{x-y}^{n+sp}}
			  }[y]
		  }.
	  \end{equation}
	  Furthermore, assume that $K$ is strictly convex.
	  \begin{enumerate}[label=(\alph*)]
		  \item
		    If $p > 1$, equality holds in \eqref{eq:fctsteiner} if and only if there exists $c \in \Real$ such that for a.e. $x' \in \Real^{n-1}$
			  \begin{equation*}
			    f(x',x_n) = f^{\#}(x',x_n-c) \text{ for a.e. } x_n \in \Real.
			  \end{equation*}
		  \item
		    If $p = 1$, equality holds in \eqref{eq:fctsteiner} if and only if for almost every $\tau > 0$ there exists $c_\tau \in \Real$ such that for a.e. $x' \in \Real^{n-1}$ the level sets $\simpleset{x_n: f(x',x_n) > \tau}$ are equivalent to intervals centered around $c_\tau$.
	  \end{enumerate}
  \end{theorem}

  We postpone the proof to the end of this section.

  From this theorem we deduce a Steiner inequality for anisotropic fractional perimeters.
  The equality cases of this fractional Steiner inequality are different from those of the classical Steiner inequality even in the isotropic case where $K$ is the Euclidean unit ball.
  If there is equality in the classical Steiner inequality $P(E) \ge P(E^{\#})$, then almost all slices $E_{x'}$ with $x' \in \Real^{n-1}$ are equivalent to intervals (cf. \cite[Theorem 14.4]{maggi}).
  However, this condition is not sufficient for equality.
  In the following fractional version equality holds precisely for sets for which almost all slices are equivalent to intervals centered around the same point.
  
  \begin{corollary}[Steiner inequality for anisotropic fractional perimeters]
    \label{steinerineq}
	  Let $E \subset \Real^n$ be a Borel set of finite measure and $K$ an origin-symmetric convex body which is symmetric with respect to the hyperplane $\simpleset{x_n=0}$.
	  If $E^{\#}$ is the Steiner symmetrization of $E$ with respect to $x_n$, then
	\begin{equation}
	  \label{eq:anisosteiner}
	  P_s(E,K) \ge P_s(E^{\#},K).
	\end{equation}
	Furthermore, assume that $K$ is strictly convex.
	Then equality holds if and only if $E$ is equivalent to a translate of $E^{\#}$.
  \end{corollary}

  \begin{proof}
	  The corollary easily follows from 
	  \begin{equation*}
	    P_s(E,K) = \frac 1 2 \integral[\Real^n]{
	      \integral[\Real^n]{
		\frac{|\chi_E(x)-\chi_E(y)|}{\norm[K]{x-y}^{n+s}}
		}[y]
	      }
	  \end{equation*}
	  and the case $p=1$ for $f=\chi_E$ in Theorem \ref{th:seminormsteiner}.
  \end{proof}

  The key result used in the proof of Theorem \ref{th:seminormsteiner} is the following general rearrangement inequality for functionals of the form
  \begin{equation*}
    \mathcal E[f,g] = \integral[\Real^m]{
      \integral[\Real^m]{
	J(f(x)-g(y))k(x-y)
      }[y]
    },
  \end{equation*}
  where $J$ is a non-negative convex function on $\Real$ and $k \in L^1(\Real^m)$ is symmetric decreasing.
  For the case $f=g$ it was proved by Frank \& Seiringer \cite{frankseiringer} and we will follow their methods closely in our proof.
  We point out that we need the statement in its full generality for two functions for the following reason:
  By Fubini we split the integrals in the definition of the seminorm,
  \begin{equation*}
		  \integral[\Real^n]{
			  \integral[\Real^n]{
				  \frac {|u(x)-u(y)|^p} {\norm[K]{x-y}^{n+sp}}
			  }[y]
			}  = \integral[\Real^{n-1}]{
			  \integral[\Real^{n-1}]{
			    \left( \integral[\Real]{
			      \integral[\Real]{
				\frac{|u_{x'}(x_n)-u_{y'}(y_n)|^p}{\norm[K]{(x'-y',x_n-y_n)}^{n+sp}}
			      }[y_n]
			    }[x_n] \right)
			  }[y']
			}[x'].
  \end{equation*}
  The expression in brackets depends on the sections $u_{x'}$ and $u_{y'}$ which are in general two different functions on $\Real$.
  
  

  We emphasize that for the equality cases the two functions $f$ and $g$ respectively their level sets share the \emph{same} center which plays a crucial role in the discussion of minimizers for the anisotropic fractional isoperimetric inequality.

  \begin{proposition}
	  \label{pr:energyprop}
	  Let $J$ be a non-negative, convex function on $\Real$ with $J(0) = 0$ and let $k \in L^1(\Real^m)$ be a symmetric decreasing function.
	  For  non-negative measurable functions $f$ and $g$ on $\Real^m$ define
	  \begin{equation}
		  \label{eq:energydef}
		  \curlyE[f,g] := \integral[\Real^m]{
			  \integral[\Real^m]{
			  J(f(x)-g(y)) k(x-y)
		  }[y]
	  }[x]
	  \end{equation}
	  and suppose that $|\simpleset{f > \tau}|$ and $|\simpleset{g > \tau}|$ are finite for all $\tau > 0$.
	  \begin{enumerate}
		  \item
			  The functional $\curlyE$ does not increase under symmetric decreasing rearrangement, i.e.
	  \begin{equation}
	    \label{eq:Eineq}
	    \curlyE[f,g] \ge \curlyE[f^*,g^*].
	  \end{equation}

  \item
	  Furthermore, suppose that $\curlyE[f,g] < \infty$ and that $k$ is strictly symmetric decreasing.
	  \begin{enumerate}

		  \item
		    If $J$ is strictly convex then equality in \eqref{eq:Eineq} holds if and only if there exists a point $c \in \Real^m$ such that for a.e. $x \in \Real^m$
  \begin{equation*}
	  f(x) = f^*(x-c) \quad \text{and} \quad g(x) = g^*(x-c),
  \end{equation*}
  i.e. $f$ and $g$ are symmetric decreasing around the same center $c$ almost everywhere.

  \item
    If $J(t) = |t|$ then equality in \eqref{eq:Eineq} holds if and only if the level sets $\simpleset{f > \tau}$ and $\simpleset{g > \tau}$ are equivalent to balls around the same center $c_\tau \in \Real^m$ for a.e. $\tau > 0$.
  \end{enumerate}
  \end{enumerate}
  \end{proposition}

  \begin{proof}
    Throughout the proof we assume that $\curlyE[f,g] < \infty$ since otherwise the inequality \eqref{eq:Eineq} holds trivially.
	  
	  First, we decompose $J$ into
	  \begin{equation*}
		  J = J_+ + J_-
	  \end{equation*}
	  where $J_+(t) = J(t)$ for $t \ge 0$ and $J_+(t)=0$ for $t \le 0$.
	  Correspondingly, $\curlyE$ can be decomposed into $\curlyE = \curlyE_+ + \curlyE_-$.
	  Since $\curlyE_-[f,g] = \tilde \curlyE_+[g,f]$ where the corresponding function $\tilde J_+(t) := J_-(-t)$ vanishes for $t \le 0$ we only need to show the assumptions for the functional $\curlyE_+$.
	  The proof consists of two steps:
	  In the first step we prove all assertions for bounded $f$ and $g$ and in the second step we remove the restriction that the functions are bounded.
	  In both cases, the essential tool will be Riesz's rearrangement inequality, Theorem \ref{th:riesz}.

	  {\it Step 1:} We assume first that $f$ and $g$ are bounded.
	  Since $J_+$ is convex, the right derivative $J_+'$ exists everywhere and is non-decreasing.
	  So we can express $J_+(f(x)-g(y))$ as integral via
	  \begin{equation*}
		  J_+(f(x)-g(y)) = \integral[-\infty][f(x)-g(y)] {J_+'(s)}[s] = \integral[0][\infty] {J_+'(f(x)-\tau) \chi_{\simpleset{g \le \tau}}(y)}[\tau].
	  \end{equation*}
	  By Fubini's theorem
	  \begin{equation*}
		  \curlyE_+[f,g] = \integral[0][\infty]{e_\tau^+[f,g]}[\tau]
	  \end{equation*}
	  where
	  \begin{equation*}
		  e_\tau^+[f,g] := \integral[\Real^m]{
			  \integral[\Real^m]{
				  J_+'(f(x)-\tau) k(x-y) \chi_{\simpleset{g \le \tau}}(y)
			  }[y]
		  }.
	  \end{equation*}
	  Note that we cannot apply the Riesz rearrangement inequality yet since the level sets $\simpleset{\chi_{\simpleset{g \le \tau}} > t}$ have infinite measure for $t < 1$.
	  Instead by the boundedness of $f$ and
	  \begin{equation*}
		  \integral[\Real^m]{
			  J_+'(f(x)-\tau)
		  } = \integral[\simpleset{f > \tau}]{
			  J_+'(f(x)-\tau)
		  } \le |\simpleset{f > \tau}| J_+'(\sup f) < \infty
	  \end{equation*}
	  we can split the integral in $e_\tau^+[f,g]$ using $\chi_{\simpleset{g \le \tau}}(y) = 1 - \chi_{\simpleset{g > \tau}}(y)$, so
	  \begin{equation*}
		  e_\tau^+[f,g] = \norm[L^1]{k} \integral[\Real^m]{J_+'(f(x)-\tau)} -
		  \integral[\Real^m]{
			  \integral[\Real^m]{
				  J_+'(f(x)-\tau) k(x-y) \chi_{\simpleset{g > \tau}}(y)
			  }[y]
		  }.
	  \end{equation*}
	  Since $J_+'$ is non-decreasing, the first integral does not change by replacing $f$ with $f^*$, and for the second integral Riesz's rearrangment inequality gives
	  \begin{equation*}
		  \integral[\Real^m]{
			  \integral[\Real^m]{
				  J_+'(f(x)-\tau) k(x-y) \chi_{\simpleset{g > \tau}}(y)
			  }[y]
		  } \le \integral[\Real^m]{
			  \integral[\Real^m]{
				  J_+'(f^*(x)-\tau) k(x-y) \chi_{\simpleset{g^* > \tau}}(y)
			  }[y]
		  }.
	  \end{equation*}
	  Together with the same argument for $\tilde \curlyE_+[g,f]$ this proves inequality \eqref{eq:Eineq} for bounded functions.

	  Next we settle the conditions for equality in this case:
	  For a.e. $\tau > 0$ we have
	  \begin{align}
		  \label{eq:strictrearr}
		  \integral[\Real^m]{
			  \integral[\Real^m]{
				  J_+'(f(x)-\tau) k(x-y) \chi_{\simpleset{g > \tau}}(y)
			  }[y]
		  } &= \integral[\Real^m]{
			  \integral[\Real^m]{
				  J_+'(f^*(x)-\tau) k(x-y) \chi_{\simpleset{g^* > \tau}}(y)
			  }[y]
		  }, \\
		  \label{eq:strictrearrtilde}
		  \integral[\Real^m]{
			  \integral[\Real^m]{
				  \tilde J_+'(g(x)-\tau) k(x-y) \chi_{\simpleset{f > \tau}}(y)
			  }[y]
		  } &= \integral[\Real^m]{
			  \integral[\Real^m]{
				  \tilde J_+'(g^*(x)-\tau) k(x-y) \chi_{\simpleset{f^* > \tau}}(y)
			  }[y]
		  }.
	  \end{align}
	  If we assume $k$ to be strictly decreasing, then by the equality cases in Riesz's rearrangement inequality, Theorem \ref{th:riesz}, there must exist $c_\tau, d_\tau \in \Real^m$ such that up to sets of measure zero
	  \begin{align*}
		  \text{by \eqref{eq:strictrearr}} && J_+'(f(x)-\tau) = J_+'(f^*(x-c_\tau)-\tau) \quad & \text{and} \quad \simpleset{g > \tau} = \set{x}{g^*(x-c_\tau) > \tau},  \\
		  \text{by \eqref{eq:strictrearrtilde}} && \tilde J_+'(g(x)-\tau) = \tilde J_+'(g^*(x- d_\tau)-\tau) \quad & \text{and} \quad \simpleset{f > \tau} = \set{x}{f^*(x- d_\tau) > \tau},   
		  %
	  \end{align*}
	  so the level sets of $f$ and $g$ are equivalent to balls for a.e. $\tau > 0$.
	  If furthermore $J$ is strictly convex and thus $J_+'$ and $\tilde J_+'$ are strictly increasing on $[0,\infty)$, from the first conditions in \eqref{eq:strictrearr} and \eqref{eq:strictrearrtilde} we deduce that $f(x) = f^*(x-c_\tau)$ and $g(x) = g^*(x-d_\tau)$ almost everywhere.
	    Since these equalities hold true for almost every $\tau > 0$, the centers $c_\tau$ and $d_\tau$ do not depend on $\tau$ and we simply write $c$ and $d$ for them.
	    On one hand, by $f(x) = f^* (x-c)$ the level sets of $f$ are equivalent to balls centered around $c$, but on the other hand, by the second statement in \eqref{eq:strictrearrtilde} almost all level sets are centered around $d$ which is only possible if $c=d$.
	    
	    If $J(t) = |t|$, then the first equality in \eqref{eq:strictrearr} and $J_+'(t) = \chi_{[0,\infty)}(t)$ imply that for a.e. $\tau > 0$ it holds that  $\simpleset{f > \tau} = \set{x}{f^* (x-c_\tau) > \tau}$, so the level sets are equivalent to balls centered around $c_\tau$.
	      But the second statement in \eqref{eq:strictrearrtilde} implies that these level sets are also centered around $d_\tau$ which can only happen if $c_\tau = d_\tau$.
	  \\

	  {\it Step 2:} We now remove the assumption that $f$ and $g$ are bounded.
	  We put $f_N := \min(f,N)$ for $N > 0$ and notice that $(f_N)^* = (f^*)_N =: f_N^*$ as well as $f_N \nearrow f$ pointwise as $N \to \infty$. 
	  Since for every $x, y \in \Real^m$ the expression $J_+(f_N(x) - g_N(y))$ is non-decreasing in $N$, by step 1 and the monotone convergence theorem we get the inequality
	  \begin{equation*}
		  \curlyE_+[f,g] \ge \curlyE_+[f^*,g^*].
	  \end{equation*}
	  Finally we turn our attention to the cases of equality whenever $k$ is strictly decreasing.
	  We decompose $f = f_N + f_u$ and $g = g_N + g_u$ with $f_N$ and $g_N$ defined as before and $f_u$ and $g_u$ possibly unbounded.
	  A calculation shows that
	  \begin{equation}
		  \label{eq:fmintegral}
		  \curlyE_+[f,g] = \curlyE_+[f_N,g_N] + \curlyE_+[f_u,g_u] + \integral[\Real^m]{
			  \integral[\Real^m]{
				  I_N(f_u(x),g_N(y))k(x-y)
			  }[y]
		  }
	  \end{equation}
	  where
	  \begin{equation*}
		  I_N(f,g) := J_+(f+N-g) - J_+(f) - J_+(N-g).
	  \end{equation*}
	  If we assume that $0 < N-g \le f$ then by convexity of $J_+$ we have
	  \begin{equation*}
		  \frac{J_+(N-g) - J_+(0)}{N-g} \le \frac{J_+(f+N-g)-J_+(f)}{N-g}
	  \end{equation*}
	  and an analogous inequality holds for exchanged roles of $N-g$ and $f$.
	  Using $J_+(0) = 0$ we get that $I_N(f,g) \ge 0$ for $0 \le g \le N$ and $f \ge 0$.
	  In particular, all integrals in \eqref{eq:fmintegral} are non-negative and finite.
	  Since $\simpleset{f_u > \tau} = \simpleset{f > \tau + N}$ it holds that $(f_u)^* = (f^*)_u$, so that by rearranging $f$ and $g$ all of the functions appearing on the right hand side of \eqref{eq:fmintegral} are replaced by their rearrangements.
	  We claim that the last integral in \eqref{eq:fmintegral} does not increase when replacing $f_u$ and $g_N$ by their rearrangements $(f_u)^*$ and $g_N^*$.
	  If $\curlyE[f,g] = \curlyE[f^*,g^*]$ then this would imply that $\curlyE_+[f_N,g_N] = \curlyE_+[f_N^*,g_N^*]$ for all $N > 0$ which would eventually lead to the equality cases established in step 1.

	  Finally, we prove the claim that the double integral in \eqref{eq:fmintegral} does not increase under rearrangement:
	  Since $J_+$ is convex its right derivative $J_+'$ is the distribution function of a non-negative measure $\mu$.
	  In particular, $J_+'(s) = \integral[0][s]{}[\mu(\tau)]$ and
	  \begin{equation*}
		  J_+(t) = \integral[0][\infty]{(t-\tau)_+}[\mu(\tau)].
	  \end{equation*}
	  This implies that
	  \begin{equation*}
		  I_N(f,g) = \integral[0][\infty]{\iota_{N,\tau}(f,g)}[\mu(\tau)]
	  \end{equation*}
	  where
	  \begin{equation*}
		  \iota_{N,\tau}(f,g) := (f+N-g-\tau)_+ - (f-\tau)_+ - (N-g-\tau)_+
	  \end{equation*}
	  so it suffices to prove that for all $\tau$ the double integral
	\begin{equation*}
		\integral[\Real^m]{
			\integral[\Real^m]{
				\iota_{N,\tau}(f_u(x),g_N(y)) k(x-y)
			}[y]
		}
	\end{equation*}
	does not increase under rearrangement.
	In order to apply the Riesz rearrangement inequality we write
	\begin{equation*}
		\iota_{N,\tau}(f,g) = \iota_{N,\tau}^{(1)}(f) - \iota_{N,\tau}^{(2)}(f,g)
	\end{equation*}
	where
	\begin{align*}
		\iota_{N,\tau}^{(1)}(f) &:= f - (f-\tau)_+, \\
		\iota_{N,\tau}^{(2)}(f,g) &:= f - (f+N-g-\tau)_+ + (N-g-\tau)_+ = \min(f, (g-N+\tau)_+).
	\end{align*}
	Since $\iota_{N,\tau}^{(1)}$ is bounded from above by $\tau$ and non-decreasing in $v$, and since by $|\simpleset{f > N}| < \infty$ the support of $f_u$ has finite measure, the integral
	\begin{equation*}
		\integral[\Real^m]{
			\integral[\Real^m]{
				\iota_{N,\tau}^{(1)}(f_u(x)) k(x-y)
			}[y]
		} = \norm[L^1]{k} \integral[\Real^m]{\iota_{N,\tau}^{(1)}(f_u(x))}
	\end{equation*}
	is finite and does not change under rearrangement.
	For the $\iota_{N,\tau}^{(2)}$-integral we use the representation of $\iota_{N,\tau}^{(2)}$ as a minimum and the layer-cake formula to write
	\begin{align*}
		\integral[\Real^m]{
			& \integral[\Real^m]{
				\iota_{N,\tau}^{(2)}(f_u(x),g_N(y)) k(x-y)
				}[y]
			} = \\
			& \integral[0][\infty]{ \left(
				\integral[\Real^m]{
					\integral[\Real^m]{
						\chi_{\simpleset{f_u > t}}(x) k(x-y) \chi_{\simpleset{(g_N-N+\tau)_+ > t}}(y)
						}[y]
				} \right)
			}[t].
	\end{align*}
	By the Riesz rearrangement inequality the double integral in brackets does not decrease under rearrangement which shows the claim.
  \end{proof}

  Next, we generalize the previous result to the case where symmetry of $k$ is only assumed for one of the factors in the decomposition $\Real^n = \Real^{n-1} \times \Real$ and we use Steiner symmetrization instead of full-dimensional Schwarz symmetrization.
  Although we only consider the case $f=g$ in the proof of Theorem \ref{th:seminormsteiner}, we state the next result for two possibly different functions $f$ and $g$ as this might be of independent interest.

  \begin{corollary}
	  \label{cor:lowerdim}
	  Let $J$ be a non-negative, convex function on $\Real$ with $J(0) = 0$, and let $k \in L^1(\Real^n)$ be such that $k_{x'}$ is a symmetric decreasing function on $\Real$ for every $x' \in \Real^{n-1}$.
	  For  non-negative measurable functions $f$ and $g$ on $\Real^n$ we define $\mathcal E[f,g]$ as in \eqref{eq:energydef} with integration over $\Real^n$.
	  Suppose that for a.e. $x' \in \Real^{n-1}$ the values $\mathcal L^1(\simpleset{f_{x'} > \tau})$ and $\mathcal L^1(\simpleset{g_{x'} > \tau})$ are finite for all $\tau > 0$.
	  \begin{enumerate}
		  \item
			  The functional $\mathcal E$ does not increase under Steiner symmetrization, i.e.
	  \begin{equation}
		  \label{eq:esteinersym}
		  \curlyE[f,g] \ge \curlyE[f^{\#},g^{\#}],
	  \end{equation}
	  where $f^{\#}$ denotes the Steiner symmetrization of $f$ with respect to $x_n$.

  \item
	  Furthermore, suppose that $\curlyE[f,g] < \infty$ and that $k_{x'}$ is strictly symmetric decreasing for every $x' \in \Real^{n-1}$.
	  \begin{enumerate}

		  \item
		    If $J$ is strictly convex then equality in \eqref{eq:esteinersym} holds if and only if there exists a point $c \in \Real$ such that for almost every $x' \in \Real^{n-1}$
  \begin{equation*}
    f(x',x_n) = f^{\#}(x', x_n-c) \quad \text{and} \quad g(x',x_n) = g^{\#}(x',x_n-c)
  \end{equation*}
  for a.e. $x_n \in \Real$.

  \item
    If $J(t) = |t|$ then equality in \eqref{eq:esteinersym} holds if and only if for a.e. $\tau > 0$ there exists $c_\tau \in \Real$ such that for a.e. $x' \in \Real^{n-1}$ the level sets $\simpleset{x_n: f(x',x_n) > \tau}$ and $\simpleset{x_n : g(x',x_n) > \tau}$ are equivalent to intervals around the same center $c_\tau$.
  \end{enumerate}

  \end{enumerate}

  \end{corollary}

  \begin{proof}
	  By Fubini we decompose the integration,
	  \begin{equation*}
		  \curlyE[f,g] = \integral[\Real^{n-1}]{
			  \integral[\Real^{n-1}]{
				  \left( \integral[\Real]{
						  \integral[\Real]{
							  J(f_{x'}(x_n) - g_{y'}(y_n))k_{x'-y'}(x_n-y_n)
						  }[y_n]
				  }[x_n] \right)
			  }[y']
		  }[x'],
	  \end{equation*}
	  where for the double integration in brackets we can apply Proposition \ref{pr:energyprop} to immediately see inequality \eqref{eq:esteinersym}.
	  For the discussion of equality cases we remark that equality in \eqref{eq:esteinersym} implies that for a.e. $x',y' \in \Real^{n-1}$ we have
	  \begin{equation*}
				\integral[\Real]{
						  \integral[\Real]{
							  J(f_{x'}(x_n) - g_{y'}(y_n))k_{x'-y'}(x_n-y_n)
						  }[y_n]
				  }[x_n] = \integral[\Real]{
						  \integral[\Real]{
						    J(f_{x'}^{*}(x_n) - g_{y'}^{*}(y_n))k_{x'-y'}(x_n-y_n)
						  }[y_n]
					  }[x_n].
	  \end{equation*}
	  Now observe that the centers $c_{x',y'}$ (resp. $c_{\tau,x',y'}$) obtained by the equality cases of Proposition \ref{pr:energyprop} cannot depend on $x'$ and $y'$ since for fixed $x'$ we have $f_{x'}(x_n) = f_{x'}^*(x_n - c_{x',y'})$ (resp. $\simpleset{f_{x'} > \tau}$ is centered around $c_{\tau,x',y'}$) for all $y'$ and vice versa for fixed $y'$ and $g_{y'}$.
  \end{proof}


  \begin{proof}[{Theorem \ref{th:seminormsteiner}}]
	  Since $|f(x)-f(y)| \ge | |f(x)| - |f(y)| |$ with equality if and only if $f$ is proportional to a non-negative function we assume that $f$ is non-negative throughout the proof.
	  Note that the kernel function $\norm[K]{x-y}^{-(n+sp)}$ is not integrable so first we rewrite the seminorm following an idea of Almgren \& Lieb \cite[p. 770]{almgrenlieb}:
	  \begin{equation*}
		  \integral[\Real^n]{
			  \integral[\Real^n]{
				  \frac {|f(x)-f(y)|^p} {\norm[K]{x-y}^{n+sp}}
			  }[y]
		  }  = \frac 1 {\Gamma(\frac{n+sp}{2})} \integral[0][\infty]{I_\alpha[f] \alpha^{\frac{n+sp}{2}-1}}[\alpha],
	  \end{equation*}
	  where $\Gamma$ is the Gamma function $\Gamma(x) = \integral[0][\infty]{t^{x-1} e^{-t}}[t]$ and 
	  \begin{equation*}
		  I_\alpha[f] := \integral[\Real^n]{
			  \integral[\Real^n]{
				  |f(x)-f(y)|^p e^{-\alpha \norm[K]{x-y}^2}
			  }[y]
		  }.
	  \end{equation*}
	  This can be seen by Fubini and computing
	  \begin{align*}
		  \integral[0][\infty]{
		  \alpha^{\frac{n+sp}{2}-1} e^{-\alpha \norm[K]{x-y}^2}
	  }[\alpha] = \integral[0][\infty]{
		  \frac{ t^{\frac{n+sp}{2}-1} }{\norm[K]{x-y}^{n+sp-2}} e^{-t} \cdot \frac 1 {\norm[K]{x-y}^2}
	  }[t] = \frac{\Gamma(\frac{n+sp}{2})}{\norm[K]{x-y}^{n+sp}},
	  \end{align*}
	  where we substituted $t = \alpha \norm[K]{x-y}^2$ in the first equality.
	  Now we are able to apply Corollary \ref{cor:lowerdim} with $J(t) = |t|^p$ and $k(\xi) = e^{-\alpha \norm[K]{\xi}^2}$ to $I_\alpha[f] = \curlyE[f,f]$.
  \end{proof}

  We conclude this section with the remark that a Steiner inequality analogous to \eqref{eq:anisosteiner} also holds for general nonlocal perimeters $\Per_k$ whenever the kernel function $k$ is symmetric decreasing in the $x_n$-coordinate, with the same equality cases if $k$ is strictly decreasing in the $x_n$-coordinate.


%
%
%
%


  \section{A P\'olya-Szeg{\H o} inequality for anisotropic symmetrization}


  \begin{lemma}
    \label{boxlemma}
    Let $L \subset \Real^n$ be a bounded set with $|L| > 0$. If for every $x = (x_1,\dots,x_n) \in L$ the box $[-|x_1|,|x_1|] \times \dots \times [-|x_n|, |x_n|]$ is fully contained in $L$, then $L$ is an unconditional star body.
  \end{lemma}

  \begin{proof}
    The set $L$ is star-shaped since for every $x \in L$ the line segment $[o,x]$ is a half-diagonal of the box $[-|x_1|,|x_1|] \times \dots \times [-|x_n|,|x_n|]$ which is fully contained in $L$.
    Furthermore, since $|L| > 0$ there exists a point $x \in L$ such that $x_i \neq 0$ for all $i=1,\dots,n$, so the box spanned by $x$ and consequently $L$ contains the origin in the interior.

    Next, we show that the radial function $\rho_L$ is continuous.
      For $u \in \unitsphere n$ and all $0 < \alpha < \rho_L(u)$ the point $x := (\rho_L(u)-\alpha)u$ is contained in $L$. 
      Denote by $\rho_\alpha$ the radial function of the box $[-|x_1|,|x_1|] \times \dots \times [-|x_n|,|x_n|]$ spanned by $x$.
      Then, by our assumption, the radial function $\rho_L$ of $L$ is bounded from below by $\rho_\alpha$, i.e.
      \begin{equation}
	\label{eq:radialbox}
	\rho_L(v) \ge \rho_\alpha (v) \quad \text{for all } v \in \unitsphere n.
      \end{equation}
	Suppose that $\rho_L$ is not continuous at $u$.
      Then there exists $\epsilon > 0$ such that in every neighbourhood of $u$ there is a point $v$ with
      \begin{equation}
	\label{eq:discontinuity}
	|\rho_L(u)-\rho_L(v)| \ge \epsilon.
      \end{equation}
      On the other hand, $\rho_\alpha$ is continuous so that 
      $\rho_\alpha(w) > \rho_\alpha(u) - \frac{\epsilon} 2$ 
      for every $w$ in a certain neighborhood of $u$.
      We only consider the case that $\rho_L(u) \ge \rho_L(v)+\epsilon$ in \eqref{eq:discontinuity} for a point $v$ in this neighbourhood, since for the case that $\rho_L(u) \le \rho_L(v)-\epsilon$ one can use similar arguments.
      Since $\rho_\alpha(u) = \rho_L(u)-\alpha$ putting the inequalities together yields 
      \begin{align*}
	\rho_\alpha(v) > \rho_\alpha(u) - \frac{\epsilon}{2} = (\rho_L(u)-\alpha)-\frac{\epsilon} 2 \ge \rho_L(v) + \frac{\epsilon}{2} - \alpha > \rho_L(v)
      \end{align*}
      for all $\alpha < \frac{\epsilon}{2}$ which is a contradiction to \eqref{eq:radialbox}.
  \end{proof}

  We recall that for a Borel set $E \subseteq \Real^n$ the set $E^{(1)}$ of points of density one, or Lebesgue points, is defined by
  \begin{equation*}
    E^{(1)}:= \set{x \in \Real^n}{ \lim_{r \to 0} \frac{|E \cap B(x,r)|}{|B(x,r)|} = 1},
  \end{equation*}
  where $B(x,r)$ denotes the open Euclidean ball around $x$ with radius $r$.
  Since $E^{(1)}$ differs from $E$ only on a set of measure zero we can restrict the study of the anisotropic fractional isoperimetric problem to sets consisting only of Lebesgue points.
  To establish symmetry of minimizers we need the following lemma which is stated in \cite{fusco}.

  \begin{lemma}
    \label{lem:lebpoints}
    Let $E \subseteq \Real^n$ be a Borel set such that for a.e. $x' \in \Real^{n-1}$ the section $E_{x'}$ is equivalent to an interval.
    Then the set of points of density one, $E^{(1)}$, of $E$ has the property that for every $x' \in \Real^{n-1}$ the section $(E^{(1)})_{x'}$ is an interval.
  \end{lemma}

  \begin{theorem}
    \label{th:minshape}
	  Let $K \subset \Real^n$ be an unconditional strictly convex body.
	  Then every minimizer $M \subset \Real^n$ of the anisotropic fractional isoperimetric inequality \eqref{eq:anfracisoineq} is up to translation equivalent to an unconditional star body.
 \end{theorem}

  \begin{proof}
    Since $K$ is symmetric with respect to every coordinate hyperplane $\simpleset{x_i=0}$, $i=1,\dots,n$, by the classification of equality cases in the Steiner inequality, Corollary \ref{steinerineq}, almost all sections of a minimizer $M$ in $x_i$-direction are equivalent to intervals centered around the same center $c_i$.
    By translation invariance we may assume that $c_i=0$ for all $i=1,\dots,n$ and by passing to $M^{(1)}$ by Lemma \ref{lem:lebpoints} we also may assume that \emph{all} sections in every coordinate direction are centered around $0$.
    This implies that for every $x \in M$ the box $[-|x_1|,|x_1|] \times \dots \times [-|x_n|,|x_n|]$ is fully contained in $M$. 
    Since $|M| > 0$ we can apply Lemma \ref{boxlemma} to finish the proof.
  \end{proof}

  \begin{remark}
    \label{rem:nonlocal}
    Assume that the kernel $k$ of the nonlocal perimeter $\Per_k$ is stricly symmetric decreasing in every coordinate direction.
    If the nonlocal isoperimetric problem \eqref{eq:nonlocalip} has a minimizer, then we can repeat all arguments in the proof of Theorem \ref{th:minshape} to see that each minimizer is up to translation equivalent to an unconditional star body.
  \end{remark}

  The next result, Theorem \ref{th:introps} of the introduction, is a P\'olya-Szeg{\H o} principle for anisotropic fractional perimeters where the symmetrization is carried out with respect to minimizers of the anisotropic fractional isoperimetric inequality.

\begin{proposition}
	Let $K \subset \Real^n$ be an unconditional strictly convex body and $M$ a minimizer of the anisotropic fractional isoperimetric inequality \eqref{eq:anfracisoineq}.
	Then the anisotropic rearrangement $f^M$ with respect to $M$ is well-defined, and
	\begin{equation}
	  \label{eq:psaniso}
		\integral[\Real^n]{
			\integral[\Real^n]{
				\frac{|f(x)-f(y)|}{\norm[K]{x-y}^{n+s}}
			}[y]
		}[x] \ge \integral[\Real^n]{
			\integral[\Real^n]{
				\frac{|f^M(x)-f^M(y)|}{\norm[K]{x-y}^{n+s}}
			}[y]
		}[x]
	\end{equation}
	for all $f \in L^1(\Real^n)$.
\end{proposition}

\begin{proof}
  The rearrangement with respect to $M$ yields a well-defined function, since by Theorem \ref{th:minshape} the minimizer $M$ is a star body.
  To show \eqref{eq:psaniso} we apply the coarea formula for anisotropic fractional perimeters (see Theorem \ref{th:coarea}) and get
	\begin{align*}
	\integral[\Real^n]{
			\integral[\Real^n]{
				\frac{|f(x)-f(y)|}{\norm[K]{x-y}^{n+s}}
			}[y]
		}[x] &= 2 \integral[0][\infty]{
			P_s (\simpleset{|f| > \tau},K)
		}[\tau] \\
		&\ge 2 \integral[0][\infty]{
			P_s(\simpleset{|f| > \tau}^M,K) 
		}[\tau] = \integral[\Real^n]{
			\integral[\Real^n]{
				\frac{|f^M(x)-f^M(y)|}{\norm[K]{x-y}^{n+s}}
			}[y]
		}[x].
	\end{align*}
\end{proof}

\subsection*{Acknowledgement}

The author would like to thank Monika Ludwig for helpful comments and discussions during the preparation of this paper.

\bibliographystyle{plain}
\bibliography{symmetry}
	Andreas Kreuml\\
	{\small
	Institut f\"ur Diskrete Mathematik und Geometrie\\
	Technische Universit\"at Wien\\
	Wiedner Hauptstra{\ss}e 8-10/1046\\
	1040 Vienna, Austria\\
E-mail: andreas.kreuml@tuwien.ac.at}

\end{document}